\documentclass[amssymb,12pt]{amsart}
\usepackage{amssymb}
\def\conjecture{{\bf
        \medskip
        \noindent Conjecture.
        \nopagebreak  }\begin{em}}

\def\corollary{{\bf
        \medskip
        \noindent Corollary.
        \nopagebreak  }\begin{em}}
\def\definition{{\bf
        \medskip
        \noindent Definition.
        \nopagebreak  }\begin{em}}

\def\proposition{{\bf
        \medskip
        \noindent Proposition.
        \nopagebreak  }\begin{em}}

\newcounter{teor}
\newenvironment{theorem}{\medskip {\bf Theorem
\refstepcounter{teor}\Alph{teor}} \begin{em}}{\end{em}\medskip}
\newtheorem{lemma}[equation]{Lemma}

\def\endconjecture{\end{em} \medskip}
\def\endcorollary{\end{em} \medskip}
\def\enddefinition{\end{em} \medskip}
\def\endproposition{\end{em} \medskip}

\newcommand{\CC}{{\mathbb C}}

\newcommand{\PP}{{\mathbb P}}

\newcommand{\RR}{{\mathbb R}}

\newcommand{\ZZ}{{\mathbb Z}}

\newcommand{\cF}{{\mathcal F}}

\newcommand{\cL}{{\mathcal L}}

\newcommand{\cN}{{\mathcal N}}

\newcommand{\cT}{{\mathcal T}}
\newcommand{\cU}{{\mathcal U}}

\newcommand{\cW}{{\mathcal W}}

\newcommand{\SU}{{\mathrm{SU}}}
\newcommand{\SO}{{\mathrm{SO}}}
\newcommand{\GL}{{\mathrm{GL}}}

\newcommand{\UU}{{\mathrm{U}}}

\newcommand \Id	{\mathrm{Id}}

\begin{document}

\title[Fold-Forms for Four-Folds]{Fold-Forms for Four-Folds}
\author[]{Ana Cannas da Silva}


\thanks{Supported in part by Funda\c{c}\~{a}o para a Ci\^{e}ncia e a
Tecnologia (FCT/Portugal).
Research at the IAS was supported in part by NSF grant DMS 9729992.}

\address{Department of Mathematics, Princeton University,
Princeton, NJ 08544-1000,
USA and Departamento de Matem\'atica, Instituto Superior T\'ecnico,
Avenida Rovisco Pais, 1049-001 Lisboa, Portugal}
\email{acannas@math.princeton.edu}


\begin{abstract}
This paper explains an application of Gromov's h-principle
to prove the existence,
on any orientable 4-manifold, of a folded symplectic form.
That is a closed 2-form which is symplectic
except on a separating hypersurface where the form singularities
are like the pullback of a symplectic form by a folding map.
We use the h-principle for folding maps (a theorem of Eliashberg)
and the h-principle for symplectic forms on open manifolds
(a theorem of Gromov) to show that, for orientable even-dimensional
manifolds, the existence of a stable almost complex structure
is necessary and sufficient to warrant the existence of
a folded symplectic form.
\end{abstract}

\maketitle


\section{Introduction}
\label{sec:introduction}

One says that a differential problem satisfies the h-principle
if any formal solution (i.e., a solution for the associated algebraic
problem) is homotopic to a genuine (i.e., differential) solution.
Therefore, when the h-principle holds, one may concentrate on a
purely topological question in order to prove the existence of
a differential solution.

Differential problems are equations, inequalities or, more generally,
relations~\cite{gromov:book} involving derivatives of maps.
The following are examples of problems known to satisfy the
h-principle: existence of immersions in strictly positive codimension
(theorems of Whitney~\cite{whitney:curves}, Nash~\cite{nash:isometric},
Kuiper~\cite{kuiper:isometric}, Smale~\cite{smale:classification},
Hirsch~\cite{hirsch:immersions} and Po\'enaru~\cite{poenaru:immersions}),
existence of symplectic forms on open manifolds (theorem of
Gromov~\cite{gromov:stable}, who built  the general machinery of the
h-principle as an obstruction theory for the sheaves of germs of maps)
and existence of maps whose only singularities are folds (theorem of
Eliashberg~\cite{eliashberg:foldingtype, eliashberg:surgery}).

This paper explains an application of the h-principle to prove the existence,
on any compact orientable 4-manifold, of a folded symplectic form, that is,
a closed 2-form with only fold singularities as defined below.
According to the h-principle philosophy, this proof is divided in two steps:
\begin{enumerate}
\item
show that the h-principle holds for this problem,
\item
show that a formal solution exists.
\end{enumerate}

For the first step, the basic ingredients are the h-principle for maps
whose only singularities are
folds~\cite{eliashberg:foldingtype,eliashberg:surgery} and the h-principle
for symplectic forms on open manifolds~\cite{gromov:stable}.
This combination is a shortcut based on an idea contained in the
forthcoming book by Eliashberg and Mishachev~\cite{eliashberg-mishachev:book}.
We thus avoid dealing with the h-principle in its generality.

Here is the flavor of Eliashberg's result.
Let $Z$ be a hypersurface in a manifold $M$,
that is, a codimension 1 embedded submanifold
(this is the meaning of {\em hypersurface} throughout this paper).
A map $f: M \to N$ between manifolds of the same dimension is called
a {\em $Z$-immersion} (or said to {\em fold along the submanifold $Z$})
if it is regular (i.e., its derivative is invertible) on $M \setminus Z$, and
if near any $p \in Z$ and near its image $f(p)$ there are coordinates
centered at those points where $f$ becomes
\[
   ( x_1, x_2, \ldots , x_n ) \longmapsto (x_1^2, x_2 , \ldots, x_n ) \ .
\]
A homomorphism $F: TM \to TN$ between tangent bundles is called
a {\em $Z$-monomorphism}, if it is injective on $T(M \setminus Z)$ and on
$TZ$, and if there exists a fiber involution $\tau : \cT \to \cT$ on a tubular
neighborhood $\cT$ of $Z$ whose set of fixed points is $Z$ and such
that $F \circ d\tau = F$.
The differential $df : TM \to TN$ of a $Z$-immersion is a $Z$-monomorphism.
Eliashberg~\cite{eliashberg:foldingtype} proved that, if every connected
component of $M \setminus Z$ is open, then any $Z$-monomorphism
$TM \to TN$ is homotopic (within $Z$-monomorphisms $TM \to TN$) to
the differential of a $Z$-immersion.
In the language of~\cite{gromov:book}, the theorem says that,
when $M \setminus Z$ is open, $Z$-immersions satisfy the
(everywhere $C^0$-dense) h-principle; a $Z$-monomorphism
is then called a {\em formal solution}.
For the present application, we require a more general
statement~\cite{eliashberg:surgery} dealing with foliated
target manifolds.

A {\em folded symplectic form} on a $2n$-dimensional manifold $M$ is a closed
2-form $\omega$ which is nondegenerate except on a hypersurface $Z$
called the {\em folding hypersurface} where, centered at every point $p \in Z$,
there are coordinates for $M$ adapted to $Z$ where the form $\omega$ becomes
\[
   x_1 dx_1 \wedge dx_2 + dx_3 \wedge dx_4 + \ldots + dx_{2n-1} \wedge dx_{2n} \ .
\]
The pullback of a symplectic form by a $Z$-immersion is a folded
symplectic form with folding hypersurface $Z$.

A formal solution for the problem of existence of a folded
symplectic form turns out to be a stable almost complex structure.
Let $M$ be a $2n$-dimensional manifold with a structure of complex
vector bundle on $TM \oplus \RR^2$, where $\RR^2$ denotes the trivial
rank 2 real vector bundle over $M$.
We will show that $M$ admits folded symplectic forms.

Here is how Gromov's theorem comes in.
We embed $M$ as level zero in $M \times \RR$.
The given stable almost complex structure on $M$ yields a complex
hyperplane field on $M \times \RR$ and hence an almost complex
structure on $M \times \RR^2$.
Since this manifold is open, Gromov's application of the
h-principle~\cite{gromov:stable} guarantees the existence
of a symplectic form on $M \times \RR^ 2$ inducing almost complex
structures in the same homotopy class as the given one.
Since $M \times \RR$ sits here as a codimension one submanifold,
the restriction $\omega_0$ of the symplectic form to this submanifold
has maximal rank, i.e., has exactly a one-dimensional kernel at every point.
Let $\cL$ be the one-dimensional foliation determined by the kernel $L$ of $\omega_0$.
The projection of $\omega_0$ to $T (M \times \RR) / L$
is well-defined and nondegenerate.
Suppose that we could immerse $M$ in $M \times \RR$ in a {\em good}
way, meaning that locally the composition of that immersion
with the projection to the local leaf space of $\cL$ is a
$Z$-immersion, for some hypersurface $Z$ in $M$.
Since this leaf space is symplectic, by pullback we would
obtain a folded symplectic form on $M$.
Hence, we concentrate on deforming the initial embedding at
level zero into a good immersion in order to prove:

\begin{theorem}
\label{theorA}
Let $M$ be a $2n$-dimensional manifold with a stable almost complex structure $J$.
Then $M$ admits a folded symplectic form consistent with $J$
in any degree 2 cohomology class.
\end{theorem}

The notion of consistency is explained in \S~\ref{sec:folded}.
The existence of a stable almost complex structure is
a necessary condition for the existence of
a folded symplectic form on an orientable manifold
(see \S~\ref{sec:folded}).
Theorem~A is then saying that it is also sufficient.
This contrasts with the case of a (honest) symplectic form,
for whose existence an almost complex structure is necessary,
but only sufficient if the manifold is open~\cite{gromov:stable}.
The sphere $S^6$ is a trivial example (thanks to Stokes' theorem)
and $\CC \PP^2 \# \CC \PP^2 \# \CC \PP^2$ is an important example
(thanks to Seiberg-Witten invariants~\cite{ta:invariants})
of almost complex manifolds without any symplectic form.

To produce a formal solution for 4-manifolds is easily accomplished.
Hirzebruch and Hopf~\cite{hirzebruch-hopf} showed that
the integral Stiefel-Whitney class $W_3$
vanishes for any compact orientable 4-manifold, or, in other words,
such manifolds always have stable almost complex structures.
(This is the same reason why such manifolds are
spin-c~\cite[Thm.D.2]{la-mi:spin}.)
Since we are in the stable range, it is enough to add a trivial
$\RR^2$ bundle to $TM$ for this to admit a structure of complex
vector bundle.
All this is also true when $M$ is not compact~\cite[\S 5.7]{gompf-stipsicz}.
We thus obtain the following relevant special case of Theorem~A:

\begin{theorem}
\label{theorB}
Let $M$ be an orientable 4-manifold.
Then $M$ admits a folded symplectic form consistent with any
given stable almost complex structure and in any degree 2 cohomology class.
\end{theorem}

In higher dimensions, there are plenty of orientable manifolds
which have no stable almost complex structures
($S^1 \times \SU (3) / \SO (3)$, for instance~\cite{la-mi:spin}),
and hence cannot have folded symplectic forms.
The condition $W_3 (M) =0$ is necessary and sufficient in
dimensions 6 (since the next obstruction $W_7$ vanishes for
dimensional reasons) and 8 (where Massey~\cite{massey:classesII}
proved that $W_7$ always vanishes).
According to~\cite{dessai:remarks,thomas:complex}, until 1998
it was still not known general necessary and sufficient
conditions (in terms of invariants such as characteristic
classes and the cohomology ring) for the existence of a
stable almost complex structure on manifolds of dimension $\geq 10$.

As for the contents of this paper:
\S~2 reviews folded symplectic manifolds and some
{\em folded} tangent bundles associated to them;
\S~3 describes the application of Gromov's theorem to
guarantee a symplectic form starting from a structure
of complex vector bundle;
\S~4 proves the existence of an isomorphism between
a {\em folded} tangent bundle and a suitable complex vector bundle;
\S~5 describes the application of Eliashberg's theorem
to produce folded symplectic forms;
\S~6 contains the conclusion of the proof of Theorems~A and~B.

{\noindent{\sc Acknowledgements:}}

The author wishes to thank Yasha Eliashberg, Gustavo Granja
and M\'elanie Bertelson for helpful discussions.

Y. Eliashberg has pointed out that Theorem~A can be alternatively
deduced from the Singular h-Principle Theorem of Gromov~\cite[p.112]{gromov:book}
following a hint in Exercise~(b) of page 113 in Gromov's book to overcome
the lack of microflexibility.


\section{Folded Symplectic Manifolds}
\label{sec:folded}

Let $M$ be an oriented manifold of dimension $2n$,
and let $\omega$ be a closed 2-form on $M$.
The highest wedge power $\omega^n$ is a section of the
(trivial) orientation bundle $\wedge ^{2n} T^*M$.

\begin{definition}
A {\em folded symplectic form} is a
closed 2-form $\omega$ such that $\omega ^n$ intersects
the 0-section of $\wedge ^{2n} T^*M$ tranversally,
and such that $\imath ^* \omega$ has maximal rank everywhere,
where $\imath : Z \hookrightarrow M$ is the inclusion of
the zero-locus, $Z$, of $\omega ^n$.
\end{definition}

By tranversality, $Z$ is a codimension-1 submanifold
of $M$, called the {\em folding hypersurface}.
A {\em folded symplectic manifold} is a pair $(M,\omega)$
where $\omega$ is a folded symplectic form on $M$.
The folding hypersurface $Z$ of a folded symplectic
manifold $(M,\omega)$ separates $M$ into the regions $M^+$ and $M^-$,
where the form matches or is opposite to the
given orientation, respectively.
Hence, $Z$ has a co-orientation depending on $\omega$
and on the choice of orientation on $M$.
(The notion of folded symplectic form extends to arbitrary
even-dimensional manifolds, not necessarily orientable,
but we will not deal with those in this paper.)

The Darboux theorem for folded symplectic forms states
that, if $(M,\omega)$ is a folded symplectic manifold and
$p$ is any point on the folding hypersurface $Z$, then there
is a coordinate chart $(\cU, x_1, \ldots ,x_{2n})$
centered at $p$ such that on $\cU$
\[
   \omega = x_1 dx_1 \wedge dx_2 + dx_3 \wedge dx_4
   + \ldots + dx_{2n-1} \wedge dx_{2n}
   \; \mbox{ and } \; Z \cap \cU = \{ x_1 = 0 \} \ .
\]
This follows, for instance, from a folded analogue
of Moser's trick~\cite{cgw:unfolding}.

Doubles of symplectic manifolds with
$\omega$-convex~\cite{el-gr:convex} (or $\omega$-concave)
boundary are easy examples of manifolds with folded
symplectic forms.
Simplest instances are the spheres $S^{2n}$, where
a folded symplectic form is obtained by pulling back
the standard symplectic form on $\RR^{2n}$
via the folding map $S^{2n} \to D^{2n}$.

Starting in dimension 4, folded symplectic forms are
not generic in the set of closed 2-forms.
Let $M$ be a (compact) oriented 4-manifold,
and let $\omega$ be a closed 2-form on $M$.
If $\gamma$ is a given volume form on $M$, then
$\omega \wedge \omega = f \gamma$ for some $f \in C^\infty (M)$.
A generic $\omega$~\cite{ma:formes} is never 0,
has rank 2 on a (compact) codimension-1 submanifold, $Z$, and
is nondegenerate elsewhere.
The hypersurface $Z$ is the 0-locus of $f$.
Its complement $M \setminus Z$ is the disjoint union of the sets
$M^+ = \{ f > 0 \}$ where $\omega$ matches the given orientation and
$M^- = \{ f < 0 \}$ where $\omega$ induces the opposite orientation.
For $\omega$ to be folded symplectic, we would need that
$TZ$ and the rank 2 bundle over $Z$ given by $\ker \omega$
intersect transversally as subbundles of $TM |_Z$.
Yet generically $\omega$ is not folded symplectic,
since its restriction to $Z$ vanishes along
some codimension-2 submanifold $C$ (a union of circles),
where $\ker \omega$ is contained in $TZ$~\cite{ma:formes}.
Although a generic 2-form on a 3-manifold
vanishes only at isolated points,
here the 3-manifold already depends on the 2-form.
Moreover, generically there are isolated {\em parabolic}
points on those lines (circles),
where the tangent space to those lines is contained in $\ker \omega$.
There is at least one continuous family of inequivalent neighborhoods
of parabolic points~\cite{arnold-givental,golubitsky-tischler}.

Now let $M$ be an $m$-dimensional manifold
with a separating hypersurface $Z$.
For instance, $M$ could be an oriented manifold equipped with
a folded symplectic form, and $Z$ its folding hypersurface.

The complement $M \setminus Z$ is the disjoint union
of open sets $M^+$ and $M^-$.
Over $Z$, the tangent bundle has a trivial line subbundle $V$,
spanned by a vector field transverse to $Z$ pointing from $M^-$ to $M^+$.
The quotient $TM / V$ is isomorphic to $TZ$, so that $TM |_Z \simeq TZ \oplus V$.

\begin{definition}
The {\em $Z$-tangent bundle} of $M$
is the rank $m$ real vector bundle ${^Z TM}$ over $M$ obtained
by gluing $TM |_{M \setminus M^-}$ to $TM |_{M \setminus M^+}$
by the constant diagonal map $\Id \oplus (-1) : Z \to \GL (TZ \oplus V)$.
\end{definition}

There are analytic and algebraic approaches to ${^Z TM}$,
which enhance its geometry~\cite{cgw:unfolding}.
From its definition it follows that:

\begin{lemma}
\label{lem:stableisom}
Let $M$ be an $m$-dimensional manifold with
a separating hypersurface $Z$.
Then there is an isomorphism of real vector bundles
\[
   TM \oplus \RR \simeq {^Z TM} \oplus \RR \ .
\]
\end{lemma}

A {\em complex structure} on a vector bundle $E$ over a manifold $M$
is a bundle homomorphism $J: E \to E$ such that $J^2 = - \Id$.
If $E$ is an orientable rank $2m$ vector bundle,
the existence of a complex structure on $E$ is equivalent
to the existence of a section of the associated
$(\SO (2m) / \UU (m))$-bundle.
A {\em stable complex structure} on a vector bundle $E$ over $M$
is an equivalence class of complex structures on the vector
bundles $E \oplus \RR^k$ ($k \in \ZZ_0 ^+$), two complex
structures, $J_1$ on $E \oplus \RR^{k_1}$ and $J_2$ on
$E \oplus \RR^{k_2}$, being {\em equivalent} when there exist
$m_1, m_2 \in \ZZ_0 ^+$ such that
$((E \oplus \RR^{k_1}) \oplus \CC^{m_1},J_1 \oplus i)$ and
$((E \oplus \RR^{k_2}) \oplus \CC^{m_2},J_2 \oplus i)$
are isomorphic complex vector bundles.
A {\em stable almost complex structure} on $M$ is a
stable complex structure on $TM$.

The $Z$-tangent bundle for the folding hypersurface
$Z$ of a folded symplectic form $\omega$ has a
canonical complex structure $J_0$~\cite{cgw:unfolding}
{\em consistent} with $\omega$.
We say that a folded symplectic form $\omega$ is {\em consistent}
with a stable almost complex structure on $M$ if
$({^Z TM} \oplus \CC , J_0 \oplus i)$ belongs to the given
equivalence class of complex structures on $TM \oplus \RR^{2k}$,
$k \in \ZZ_0^+$.


\section{First Instance of the h-Principle}
\label{sec:first}

Let $M$ be a $2n$-dimensional manifold with a stable almost complex structure.
The homotopy groups $\Pi_q (\SO (2m) / \UU (m))$ are isomorphic
for fixed $q$ and variable $m$ such that $q < 2m-1$
(this is the so-called {\em stable range}~\cite{massey:obstructions}).
Hence, if there exists a complex structure on $TM \oplus \RR^{2k}$,
then there exists a complex structure on $TM \oplus \RR^2$.

Let $J$ be a complex structure on $TM \oplus \RR^2$.
Let
\[
\begin{array}{rrclcrrcl}
   i : & M & \hookrightarrow & M \times \RR
   & \; \mbox{ and } \; &
   \pi : & M \times \RR & \twoheadrightarrow & M \\
   & p & \mapsto & (p,0) & &
   & (p,t) & \mapsto & p
\end{array}
\]
be the embedding at level zero, and the projection to the first factor.
By pullback, $i$ induces an isomorphism in cohomology.

Via the identification $T(M \times \RR) \simeq \pi^* (TM) \oplus \RR$,
the structure $J$ induces a structure of complex vector bundle,
still called $J$, on $T(M \times \RR) \oplus \RR \simeq \pi^* (TM) \oplus \CC$.
Then the complex subbundle
\[
   H_0 = T (M \times \RR) \cap J (T (M \times \RR))
   \subset T(M \times \RR) \oplus \RR
\]
is a complex hyperplane field over $M \times \RR$.
Let $\omega_1$ be a 2-form of maximal rank in $M \times \RR$
{\em compatible with $J$}, that is,
\[
   \omega_1 (u,v) = g(J u,v)\ , \ \forall u,v \in H_0 \ ,
   \mbox{ and } \omega_1 (u , \cdot) = 0 \ , \ \forall u \in H_0^\perp \ ,
\]
for some riemannian metric $g$ on $TM \times \RR$, where
$H_0^\perp$ denotes the orthocomplement of $H_0$ with respect to $g$.
A {\em regular homotopy} of two 2-forms of maximal rank
is a homotopy within 2-forms of maximal rank.

\begin{lemma}
\label{lem:first}
Let $M$ be a manifold with a structure $J$
of complex vector bundle on $TM \oplus \RR^2$.
Then  there exists in $M \times \RR$ a closed 2-form of
maximal rank in any degree 2 cohomology class,
which is regularly homotopic to any 2-form of maximal
rank compatible with $J$.
\end{lemma}

This is an immediate consequence of the
following proposition which was originally proved by
McDuff~\cite{mcduff:application}.
The proof below is taken from
Eliashberg-Mishachev~\cite{eliashberg-mishachev:book}.
We reproduce it since this result is not as widely known as
the other applications of the h-principle and since the idea in
this proof is crucial for the present paper's strategy.
The key to this proof is Gromov's theorem~\cite{gromov:stable} saying
that, for every degree 2 cohomology class on any open manifold,
any nondegenerate 2-form is regularly homotopic to a symplectic
form in that class; moreover, if two symplectic forms are
regularly homotopic, then they are homotopic within symplectic forms.
Recall that a manifold is {\em open} if there are no closed manifolds
(i.e., compact and without boundary) among its connected components.

\begin{proposition}\cite{mcduff:application}
For any 2-form of maximal rank on an odd-dimen\-sional
manifold and any degree 2 cohomology class, there
exists a closed 2-form of maximal rank in that class which
is regularly homotopic to the given form.
\end{proposition}

\begin{proof}
Let $\omega_1$ be a 2-form of maximal rank on a
$(2n+1)$-dimensional manifold $N$ and let $\alpha$ be
a degree 2 cohomology class in $N$.
By homotopy, the projection to the first factor
$\pi : N \times \RR \to N$ induces an isomorphism in cohomology.

If $N$ is orientable, then $\omega_1$ extends in a homotopically
unique way compatible with orientations to a
nondegenerate 2-form, $\omega_2$, in $N \times \RR$.
Gromov's result~\cite{gromov:stable} cited above
guarantees the existence, in the class $\pi^* \alpha$, of a
homotopically unique symplectic form $\omega_3$ in
$N \times \RR$ regularly homotopic to $\omega_2$.
The restriction of $\omega_3$ to the zero level $M$ is
a closed 2-form of maximal rank.

If $N$ is not orientable, we replace $N \times \RR$ in the
previous argument by the total space of the real line bundle
given by the kernel of $\omega_1$.
\end{proof}


\section{Vector Bundle Isomorphism}
\label{sec:isomorphism}

Let $\widetilde \omega$ be a closed 2-form of
maximal rank in $M \times \RR$, and let $L$
be the line field on $M \times \RR$ given by the kernel of
$\widetilde \omega$ at each point.
By orientability of $M$, the line bundle $L$ is trivializable.
Let $\cL$ be the 1-dimensional foliation
corresponding to $L$.
Choose a complementary hyperplane field $H$ so that
$T(M \times \RR) \simeq H \oplus L$.

Let $Z_0$ be a separating hypersurface in $M$
with a coorientation.
Since by Lemma~\ref{lem:stableisom} we have that
\[
   {^{Z_0} TM} \oplus \RR \simeq TM \oplus \RR
   \simeq i^* (H \oplus L) \ ,
\]
the restriction $i^* H$ is stably isomorphic to ${^{Z_0} TM}$.
The Stiefel-Whitney classes are stable invariants,
and the mod $2$ reduction of the Euler class of an orientable
rank $m$ real vector bundle $E$ coincides with the $m$th
Stiefel-Whitney class of $E$ (see, for instance,~\cite{milnor-stasheff}).
Therefore the Euler numbers (i.e., the evaluations of the
Euler classes over the fundamental homology class) of $i^* H$
and of ${^{Z_0} TM}$ differ by an even integer, let us say
\[
   \chi (i^* H) = \chi ({^{Z_0} TM}) + 2k \ .
\]

If two stably isomorphic orientable rank $2n$ real vector
bundles over an $2n$-dimensional connected manifold have
the same Euler number, then they are isomorphic.
This was contained in the work of Dold and Whitney when
the base is a 4-manifold~\cite{dold-whitney}.
In general, this follows from observing in the diagram
\[
\begin{array}{rcl}
   & & S^{2n} \hookrightarrow \SO / \SO(2n) \\
   & \nearrow & \downarrow \\
   M^{2n} & \rightrightarrows & \mathrm{BSO} (2n) \\
   & \searrow & \downarrow \\
   & & \mathrm{BSO}
\end{array}
\]
that the fiber $\SO / \SO(2n)$ of ${\mathrm{BSO}} (2n) \to {\mathrm{BSO}}$
is $(2n-1)$-connected, that $[M^{2n}, S^{2n}]$ $\simeq \ZZ$
where the homotopy type is detected by the degree,
and that the pullback of the Euler class to $S^{2n}$ is nontrivial
(since $S^{2n} \to \mathrm{BSO} (2n)$ is the classifying map for $TS^{2n}$).

Consider the following operation on rank $m$ real vector
bundles over $m$-dimensional manifolds.
If $E$ is such a bundle and $D^m$ is a small disk in the
base manifold $M$, let $E \sharp TS^m$ be the bundle
obtained by gluing $E |_{M \setminus {\mathrm{Int}} \, D^m}$
to the trivial bundle $\RR^m$ over $D^m$ by
the characteristic map of $TS^m$, i.e., by the map
$S^{m-1} \to \SO (m)$ which characterizes the
tangent bundle of $S^m$ as the gluing over the equator
of northern and southern trivial bundles~\cite[\S 18.1]{steenrod}.
For an integer $k$, the bundle $E \sharp k TS^m$
is built analogously by taking the $k$th power of
the characteristic map of $S^m$.
By counting with orientations the vanishing points of
a section transverse to zero, we see that $E \sharp k TS^m$
has Euler characteristic $\chi (E) + 2k$.
We conclude that
\[
   i^* H \simeq {^{Z_0} TM} \sharp k TS^{2n} \ .
\]
For $k$ positive, let $Z$ be the union of $Z_0$
with $k$ homologically trivial spheres $S^n$
contained in the negative part of $M \setminus Z_0$
with respect to the given coorientation.
For $k$ negative, define $Z$ similarly but with the
spheres in the positive part of $M \setminus Z_0$.
It follows from the computations in~\cite[\S 3.9]{eliashberg:foldingtype}
that $i^* H$ and ${^Z TM}$ have the same Euler number,
and hence are isomorphic.
It is possible to start from the empty hypersurface,
in which case a coorientation is not defined.
Yet the same argument holds by taking
$Z$ to be a union of spheres (as many as half of the
absolute value of the difference of the Euler numbers
of $TM$ and of $i^* H$) whose coorientation is
determined by the sign of $k$ above.
We have thus proved the following:

\begin{lemma}
\label{lem:isomorphism}
Let $H$ be a coorientable hyperplane field in $M \times \RR$
and $i : M \hookrightarrow M \times \RR$ the inclusion
at level zero.
The restriction $i^* H$ is isomorphic to ${^Z TM}$,
where $Z$ is a separating hypersurface as described
in the previous paragraph.
\end{lemma}


\section{Second Instance of the h-Principle}
\label{sec:second}

Throughout this section, let $M$ be an $m$-dimensional manifold
with a hypersurface $Z$, and let $N$ be an $(m+1)$-dimensional
manifold with a 1-dimensional foliation $\cL$.
The following notions are due to Eliashberg~\cite{eliashberg:surgery}.

\begin{definition}
   A map $f: M \to N$ is a {\em $Z$-immersion relative to $\cL$},
if near any point $p \in M \setminus Z$ there are coordinates
$y_1 , \ldots , y_{m+1}$ in $N$ adapted to the foliation (i.e., each leaf
is a level set of the first $m$ coordinates) where the induced map
to each level set of $y_{m+1}$ is regular, and if near any $p \in Z$
and near its image there are coordinates centered at those points
and adapted to the foliation where $f$ becomes
\[
   (x_1 , x_2 , \ldots , x_m ) \longmapsto
   (x_1^ 2 , x_2 , \ldots , x_m , 0) \ .
\]
\end{definition}

In the adapted coordinates $x_i$, the hypersurface $Z$ is given by $x_1 = 0$.
Loosely speaking, a $Z$-immersion relative to $\cL$ is a
$Z$-immersion to the leaf space of $\cL$.
The definition extends to higher-dimensional foliations
whose codimension is equal to the dimension of $M$.

\begin{lemma}
\label{lem:zimmersion}
   Let $\widetilde \omega$ be a closed 2-form of maximal rank in $N$
whose kernel is the tangent space to the leaves of $\cL$.
If $f: M \to N$ is a $Z$-immersion relative to $\cL$, then
$f^* \widetilde \omega$ is a folded symplectic form on $M$
with folding hypersurface $Z$.
\end{lemma}

The reason is simply that the form $\widetilde \omega$ induces
a symplectic form in the local leaf spaces and that the composition
of $f$ with the local quotient maps is a $Z$-immersion.

\begin{proof}
   Let $p \in M$.
There is a neighborhood $\cU$ of $f(p)$ where we have a trivialization
$\cU \simeq \cF _\cU \times \cL_\cU$,
given in local coordinates centered at $f(p)$ by
$(x_1 , \ldots , x_{m+1} ) \mapsto ((x_1, \ldots , x_m), x_{m+1})$,
the set $\cF_\cU$ being a leaf space (say the level zero of $x_{m+1}$),
and $\cL_\cU$ a typical leaf (say the level zero of $(x_1, \ldots , x_m)$).
The restriction of $\widetilde \omega$ to $\cF_\cU$ is a
symplectic form, $\omega_\cU$.
The composition $g_\cU : f^{-1} (\cU) \to \cF_\cU$ of $f$
with the projection to $\cF_\cU$ is a $(Z \cap \cU)$-immersion,
so that $g_\cU^* \omega_\cU$ is a folded symplectic form
with folding hypersurface $Z \cap \cU$.
The result follows from the fact that $f^* \widetilde \omega$
on $f^{-1} (\cU)$ coincides with $g_\cU ^* \omega_\cU$.
\end{proof}

We now turn to the formal analogue of a $Z$-immersion.

\begin{definition}
   A bundle map $F: TM \to TN$ is a {\em $Z$-monomorphism
relative to $\cL$}, if $F |_{T(M \setminus Z)}$ is transverse to $\cL$,
and if each $p \in Z$ admits a neighborhood
$\cU$ where $F |_{T \cU}$ is the differential of some
$(Z \cap \cU)$-immersion relative to $\cL$.
\end{definition}

The following lemma is a direct consequence of Eliashberg's result
in~\cite[\S 6.3]{eliashberg:surgery}, where he extends to the case
of foliations the result described in the introduction.

\begin{lemma}
\label{lem:second}
Let $N = M \times \RR$ be equipped with a decomposition
$TN \simeq H \oplus L$, where $L$ is a line field, and
let $\cL$ be the corresponding 1-dimensional foliation.
Let the hypersurface $Z$ be such that every
connected component of $M \setminus Z$ is open.
Then, for every $Z$-monomorphism $F: TM \to TN$ relative to $\cL$,
there exists a $Z$-immersion $f: M \to N$ relative to $\cL$
whose differential $df$ is homotopic to $F$ through
$Z$-monomorphisms relative to $\cL$.
\end{lemma}

Part of the work to prove Theorem~A consists in showing
a (general) procedure to deform by homotopy a weaker bundle map
into a $Z$-monomorphism relative to $\cL$.
The weaker map is of the following type:

\begin{definition}
   A bundle map $F: TM \to TN \simeq H \oplus L$ is a
{\em $Z$-monomorphism relative to $L$},
if $\pi_L \circ F |_{T(M \setminus Z)}$ and $\pi_L \circ F |_{TZ}$
are fiberwise injective, $\pi_L : TN \to H$ being the projection
along $L$, and if there is a tubular neighborhood
$\cT$ of $Z$ in $M$, with a fiber involution $\tau : \cT \to \cT$
whose set of fixed points is $Z$, where $F \circ d\tau = F$.
\end{definition}


\section{Conclusion of the Proof}
\label{sec:proof}

Let $M$ be a compact $2n$-dimensional manifold with
a stable almost complex structure $J$.
Then $J$ is representable by a structure of complex
vector bundle on $TM \oplus \RR^2$, and any two such representatives
are isomorphic, by Bott periodicity~\cite{bott:stable}.
Let $N = M \times \RR$ and denote still by $J$
an induced structure of complex vector
bundle on $TN \oplus \RR$ as in \S~\ref{sec:first}.

By Lemma~\ref{lem:first},
there exists on $N$, in any degree 2 cohomology class,
a closed 2-form $\widetilde \omega$ of maximal rank
compatible with $J$.
Let $\widetilde \omega$ be such a form and let
$L$ be the line field given by its kernel,
with associated foliation $\cL$.

By Lemma~\ref{lem:zimmersion},
the existence of a folded symplectic form on $M$
with some folding hypersurface $Z$ is guaranteed by the
existence of a $Z$-immersion $f: M \to N$ relative to $\cL$.
We will seek such a $Z$-immersion which is homotopic to the
embedding at level zero $i : M \hookrightarrow N$,
so that $f^* = i^*$ in cohomology.
If $M$ is connected and $Z$ is nonempty,
then $M \setminus Z$ is open.

By Lemma~\ref{lem:second},
in order to produce a $Z$-immersion $f$ relative to $\cL$
for $M \setminus Z$ open, it suffices to show that there exists
a $Z$-monomorphism $F : TM \to TN$ relative to $\cL$.
So that $f$ is homotopic to $i$, we search for an $F$
covering a map $M \to N$ homotopic to $i$.

By Lemma~\ref{lem:isomorphism}, we have
a vector bundle isomorphism $F_0 : {^Z TM} \to i^* H$
for some hypersurface $Z$, which may be chosen so that
each connected component of $M \setminus Z$ is open.

The map $F_0$ may be translated into a fiberwise injective bundle map
$F_1 : {^Z TM} \to H$ covering the immersion $i : M \to N$.
This map guarantees the existence of a (canonically
unique up to homotopy) almost $Z$-monomorphism
$F_2 : TM \to H \oplus L$ relative to $L$,
still covering $i$, defined by the following recipe:

Choose a trivial line bundle $V$ over $Z$ spanned by a
vector field on $M$ transverse to $Z$
pointing from $M^-$ to $M^+$.
The quotient ${^Z TM} / V$ is isomorphic to $TZ$,
so that ${^Z TM} |_Z \simeq TZ \oplus V$.
We obtain $TM$ by gluing ${^Z TM} |_{M \setminus M^-}$
to ${^Z TM} |_{M \setminus M^+}$ by the constant
diagonal map $\Id \oplus (-1) : Z \to \GL (TZ \oplus V)$.
Using this recovery of $TM$ from ${^Z TM}$,
we may define $F_2$ equal to $F_1 \oplus 0$ outside a tubular
neighborhood $\cT$ of $Z$ in $M$, and on $\cT$ set
\[
   F_2 ( u \oplus v ) = F_1 ( u \oplus \psi v ) \oplus 0 \ ,
\]
with respect to the decomposition
${^Z TM} | _\cT \simeq \pi^* (TZ) \oplus \pi^* V$,
where $\pi: \cT \to Z$ is the tubular projection,
and $\psi : \cT \to [0,1]$ is equal to 1 outside a
narrower tubular neighborhood of $Z$ and vanishes
exactly over $Z$.
By choosing $\psi$ symmetric with respect to an
involution $\tau : \cT \to \cT$ whose set of
fixed points is $Z$, we obtain $F_2$ invariant under $\tau$.

For each $p \in Z$, choose a connected neighborhood $\cU$ whose
image $i (\cU)$ is contained in a connected trivialization
$\cN_\cU \simeq \cF_\cU \times \cL_\cU$ of the foliation $\cL$,
the set $\cF_\cU$ being a local leaf space and $\cL_\cU$ a leaf segment.
Let $\pi_\cU : \cN_\cU \to \cF_\cU$ be the projection to the first factor.
The composition $F_{2,\cU} = d \pi_\cU \circ F_2 |_\cU: T\cU \to T\cF_\cU$
is a $(Z \cap \cU)$-monomorphism.
\[
\begin{array}{rcl}
   & & T \cN_\cU \\
   & {^{F_2} \nearrow} \phantom{99999} & \downarrow _{d \pi_\cU} \\
   T \cU & \stackrel{F_{2,\cU}}{\longrightarrow} & T\cF_\cU
\end{array}
\]
By~\cite[\S 2.2]{eliashberg:foldingtype},
the composition $F_{2,\cU}$ is homotopic,
through $(Z\cap \cU)$-monomorphisms, to the differential
$d g_\cU$ of a $Z$-immersion $g_\cU : \cU \to \cF_\cU$.
Moreover, if over a closed subset $\cW \subset \cU$,
the composition $F_{2,\cU}$ was already
the differential of a map, then there is a homotopy
which is constant on $\cW$.
Let $G_t : T \cU \to T \cF_\cU$, $1 \leq t \leq 2$,
be a homotopy such that $G_1 = dg_\cU$ and $G_2 = F_{2,\cU}$.

Choose a $(Z \cap \cU)$-immersion $\widetilde g_\cU : \cU \to \cN _\cU$
relative to $\cL$ such that $\pi_\cU \circ \widetilde g_\cU = g_\cU$.
We can always pick a $\widetilde g_\cU$ extending a sensible
preassigned lift over a closed subset $\cW$ of $\cU$.

By the covering homotopy property for the fibering
$T \cN _\cU \to T\cF _\cU$, there is a lifted homotopy
$\widetilde G_t : T \cU \to T \cN_\cU$, $1 \leq t \leq 2$,
through $Z$-monomorphisms relative to $L$ such that
$\widetilde G_1 = d\widetilde g_\cU$ and
$d \pi_\cU \circ \widetilde G_t = G_t$ for all $t$.
If $G_t$ was constant on a closed subset $\cW$,
then we may choose $\widetilde G_t$ also constant on $\cW$.
\[
\begin{array}{rcl}
   & & T \cN_\cU \\
   & {^{\widetilde G_t} \nearrow} \phantom{99999} & \downarrow _{d \pi_\cU} \\
   T \cU & \stackrel{G_t}{\longrightarrow} & T\cF_\cU
\end{array}
\]
Since $d \pi_\cU \circ \widetilde G_2 = G_2 = F_{2,\cU} = d \pi_\cU \circ F_2$,
the difference $\widetilde G_2 - F_2$ takes values in $L = \ker d\pi_\cU$.
By fiberwise homotopy, we may deform the vertical
component of $\widetilde G_2$ to make it equal to $F_2$.
Without loss of generality, we hence assume that $\widetilde G_t$
also satisfies $\widetilde G_2 = F_2$, and that all maps
are invariant with respect to the same involution $\tau$.

Take a riemannian metric symmetric with respect to $\tau$.
For a point $p \in Z$, choose spherical neighborhoods
$\cU_1$ and $\cU_2$ in $\cT$, consisting of points
at a riemannian distance less than $\varepsilon$ and
$4 \varepsilon$ from $p$, with $\varepsilon > 0$ small
enough for the exponential map to be injective and for
the closure of $\cU_2$ to be contained in the neighborhood
$\cU$ above.
Choose a smooth function $\rho : \cU_2 \to [1,2]$
satisfying $\rho (q) =2$ if the distance from $p$ to $q$
is greater than $3 \varepsilon$, and $\rho (q) =1$
if the distance from $p$ to $q$ is less than $2 \varepsilon$.
Define $F_3 : TM \to TN$ by
\[
   F_3 = \begin{cases}
   F_2 & \quad \mbox{ on } M \setminus \cU_2 \\
   \widetilde G_{\rho (q)} & \quad \mbox{ over points }
   q \in \cU_2 \setminus \cU_1 \\
   d \widetilde g_\cU & \quad \mbox{ on }\cU_1 \ . \end{cases}
\]
Then $F_3$ is a $Z$-monomorphism with respect to $L$
whose restriction to $\cU_1$ is the differential of a
$(Z \cap \cU_1)$-immersion relative to $\cL$.

Since $Z$ is compact, take a subcover of $Z$ in $M$ by a finite
number of the $\cU_1$'s.
Apply iteratively the construction of the previous paragraph
to an ordering of the $\cU_1$'s,
starting first from $F_2$ and then from its replacements $F_3$, etc.
At each stage, the homotopy should be taken constant over
the closure $\cW$ of the previous $\cU_1$'s.

We have thus concluded the proof of Theorem~A
in the compact case by showing the existence of a
$Z$-monomorphism relative to $\cL$ covering
a map homotopic to $i$.

\begin{remark}
If $M$ is a compact oriented 2-dimensional manifold,
folded symplectic forms on $M$ are generic 2-forms.
The cohomology class of a 2-form is determined by its total integral.
The isomorphism classes of complex structures on $TM \oplus \RR^2$
are determined by the Euler number, which is an even integer.
By changing $Z$ as in \S~\ref{sec:isomorphism}, any even number
may be obtained as Euler number for ${^Z TM}$,
thus fitting any given stable complex structure.
Let $\omega$ be a 2-form which vanishes transversally
on an appropriate $Z$.
By changing the values of $\omega$ over $M \setminus Z$,
any real number may be obtained as total integral of $\omega$.
Hence, Theorem~A holds easily (and not interestingly)
for compact 2-manifolds.
\end{remark}

For the noncompact case, a statement stronger than Theorem~A is true.
If a $2n$-dimensional manifold $M$ is orientable, connected,
not compact and $TM \oplus \RR^2$ has a complex structure,
then $M$ has an almost complex structure because it retracts
to a $(2n-1)$-dimensional cell complex~\cite[Thm.8.1]{milnor:hcobordism}
and $\Pi_q (\SO (2n) / \UU (n)) \simeq \Pi_q (\SO (2n+2) / \UU (n+1))$
for $q \leq 2n-2$.
By Gromov's theorem~\cite{gromov:stable}, $M$ admits a
compatible symplectic form in any degree 2 cohomology class.

Let $E$ be a rank $2m$ oriented real bundle over $M$.
The condition $W_3 (E) =0$ ensures the existence over the
3-skeleton of $M$ of a section for the associated
$(\SO (2m) / \UU (m))$-bundle.
By Bott's periodicity, $\Pi_q (\SO (6) / \UU (3)) =0$ for $q < 5$.
Therefore, the Hirzebruch-Hopf fact~\cite{hirzebruch-hopf}
that $W_3 (M) =0$ for any orientable 4-manifold,
asserts the existence of a stable complex structure
on any such manifold.


\end{document}